\documentclass[11pt]{amsart}

\usepackage[utf8]{inputenc}

\usepackage{amssymb, graphicx}

\usepackage{amsfonts}
\usepackage{amsmath}
\usepackage{latexsym}
\usepackage{amscd}
\usepackage{verbatim}

\newcommand{\tnz}{\otimes}

\allowdisplaybreaks[1]

\newtheorem{theorem}{Theorem}[section]
\newtheorem{proposition}[theorem]{Proposition}
\newtheorem{lemma}[theorem]{Lemma}

\theoremstyle{remark}

\newcommand{\N}{\mathbb{N}}

\newcommand{\iso}{\cong}

\newcommand{\B}{\tilde{M}}
\newcommand{\Ce}{\mathcal{C}}
\newcommand{\CC}{k(\xi)}
\newcommand{\K}{k[\xi]}

\newcommand{\cha}{\mathrm{char}}

\newcommand{\adj}{{\rm adj}}

\begin{document}
\title[Functional identities of one variable]{Functional identities of one variable}

\author{Matej Bre\v sar}
\author{\v Spela \v Spenko}
\address{M. Bre\v sar,  Faculty of Mathematics and Physics,  University of Ljubljana,
 and Faculty of Natural Sciences and Mathematics, University
of Maribor, Slovenia} \email{matej.bresar@fmf.uni-lj.si}
\address{\v S. \v Spenko,  Institute of  Mathematics, Physics, and Mechanics,  Ljubljana, Slovenia} \email{spela.spenko@imfm.si}

\begin{abstract}
Let $A$ be a centrally closed prime algebra over a characteristic $0$ field $k$, and let $q:A\to A$ be the trace of a $d$-linear map (i.e., $q(x)=M(x,\ldots,x)$ where $M:A^d\to A$ is a  $d$-linear map).
If $[q(x),x]=0$ for every $x\in A$, then $q$ is of the form $q(x) =\sum_{i=0}^{d} \mu_i(x)x^i$ where each $\mu_i$ is the trace of a  $(d-i)$-linear map from $A$ into $k$. For infinite dimensional algebras
and algebras of dimension  $>d^2$ this was proved by Lee, Lin, Wang, and Wong in 1997.
In this paper we cover the remaining case where the dimension is $ \le d^2$. Using this result we are able to handle general functional identities of one variable on $A$; more specifically,  
 we describe  the traces of  $d$-linear maps $q_i:A\to A$ that satisfy $\sum_{i=0}^m x^i q_i(x)x^{m-i}\in k$ for every $x\in A$.
\end{abstract}

\thanks{2010 {\em Math. Subj. Class.} Primary 16R60. Secondary 13A50, 13B22,  16R20.}
\thanks{Supported by ARRS Grant P1-0288.}

\maketitle

\section{Introduction}

Let $R$ be a ring.
A map $q:R\to R$ is said to be  {\it commuting} if $[q(x),x]=q(x)x-xq(x)=0$ for all $x\in R$. The study of commuting maps has a long and rich history, starting with Posner's 1957 theorem stating  that there are no  nonzero commuting derivations on  noncommutative prime rings \cite{Pos}. The reader is referred to the survey article \cite{Bre00} for the general theory of commuting maps. We will be concerned with commuting traces of multilinear maps on prime algebras.
A map $q$ between additive groups (resp. vector spaces)  $A$ and $B$ is called the  trace of a $d$-additive (resp. $d$-linear)  map if there exists a $d$-additive (resp. $d$-linear) map $M:A^d\to B$ such that  $q(x)=M(x,\dots,x)$ for all $x\in A$ (for $d=0$ this should be understood as that $q$ is a constant). 
If $R$ is a prime ring (resp. algebra) and $q:R\to R$ is a commuting trace of a $d$-additive map,  is then  $q$  necessarily of a {\em standard form}, meaning that there exist traces of $(d-i)$-additive (resp. $(d-i)$-linear)  maps $\mu_i$ from $R$ into the extended centroid $C$ of $R$ such that
$q(x) =\sum_{i=0}^{d} \mu_i(x)x^i$ for all $x\in R$?  This question was initiated in 1993 by the first author who  obtained an affirmative answer   for $d=1$ \cite{Bre93'}, and also  for $d=2$   
 \cite{Bre93} provided that char$(R)\ne 2$ and   $R$ does not satisfy $S_4$, the standard polynomial identity of degree $4$. The $d=3$ case was treated, in a slightly different context, by Beidar, Martindale, and Mikhalev \cite{BMM}. 
These three results have turned out to be applicable to various problems, particularly in the Lie algebra theory, and have played a crucial role in the development of the theory of functional identities \cite{FI}.
 The next step was made in 1997 by
Lee, Lin, Wang, and Wong \cite{LLWW}, who answered the above question in affirmative for a general $d$, but under the assumption that $\cha(R)=0$ or $\cha(R)>d$ and $R$ does not satisfy the standard polynomial identity 
$S_{2d}$. The reason for the exclusion of rings satisfying polynomial identities (of low degrees) is the method of the proof; it is, therefore, natural to ask whether, by a necessarily different method, one can get rid of this assumption. For $d=2$ this has turned out to be the case
\cite{BS}. To the best of our knowledge,  the question is still open  for a general $d$ (cf. \cite[p. 377]{Bre00}
   and \cite[p. 130]{FI}). The first main goal of this paper is to give an affirmative answer for centrally closed prime algebras and traces of $d$-linear maps (Theorem \ref{krce}).   In this setting the problem can be  reduced to the case where the algebra in question is a full matrix algebra. The essence of our approach is interpreting the reduced problem in the algebra of generic matrices, which enables  to use the tools of Commutative Algebra.
   
In the second part of the paper,  in Sections \ref{s4}--\ref{s7}, we will deal with considerably more general functional identities. However,  the main issue will be the reduction to the commuting trace case.  Roughly speaking, the general theory of  functional identities deals with  expressions of the form
   $$
   \sum x_{i_1}\cdots x_{i_p} F_t( x_{j_1},\ldots, x_{j_q}) x_{k_1}\cdots x_{k_r},
   $$ 
 which are assumed  to be either $0$ or central for all elements $x_{i_1},\ldots,x_{k_r}$; here, $F_t$ are arbitrary   functions, and the goal is to describe their form \cite{FI}. One usually assumes that the ring (algebra) in question does not satisfy nontrivial polynomial identities (of low degrees), which makes the problem essentially easier. The second main goal of this paper is to study functional identities of one variable in general  centrally closed prime algebras. That is, we will consider  the situation where the summation  $\sum_{i=0}^m x^i q_i(x)x^{m-i}$ is always zero or central; here, $q_i$ will be assumed to be the traces of multilinear maps (so that after the linearization one gets an expression as above). 
 The main novelty is that we will handle  the finite dimensional case (which can be actually easily reduced to the  full matrix algebra case). We will see that the only   ``solutions" of our functional identity that do not exist already in the infinite dimensional case arise from the Cayley--Hamilton theorem; more precisely, they can be expressed through the adjugation  (Theorem \ref{MT}).

For simplicity, we assume throughout the paper that $k$ is a field with $$\cha(k)=0.$$ 
From the proofs it is evident that in some of the results we only need  $k$ to be infinite or having characteristic either $0$ or big enough.


\section{Commutative algebra preliminaries}\label{s2}

 Let
 $\xi=\{\xi_{ij}\mid 1\leq i,j\leq n\}$ be the set of commuting indeterminates.    We write  $Y$ for the generic $n\times n$ matrix $(\xi_{ij})$, and  $c(t)$ for the characteristic polynomial of  $Y$, i.e., $c(t)=\det(t-Y) \in  k[\xi][t]$ (throughout the paper, we use the same notation for scalars and scalar multiples of unity).

\begin{lemma}\label{ch}
The polynomial  $c(t)$   is  irreducible in $\CC[t]$. Accordingly,  $k[\xi][t]/(c(t))$ is an integral domain.
\end{lemma}

\begin{proof}
We consider $c=c(t)$ as a polynomial in indeterminates $t,\xi_{ij},1\leq i,j\leq n$. 
Suppose $c=fg$ for some $f,g\in k(\xi)[t]$ of positive degree in $t$. Then $c$ is reducible already in 
$k[\xi][t]$ (see, e.g., \cite[Corollary 3.5.4]{CLO}), so we may assume that $f,g\in k[\xi][t]$. Since $c$ has degree $n$, $f$ and $g$ have to be of degree strictly less than $n$. Regarding $c,f,g$   as polynomials in $t$, the constant term of $c$ equals the product of the constant terms of $f$ and $g$. Since the former equals $(-1)^n\det(Y)$, which is irreducible  \cite[pp. 7--8]{Berg}, and none of $f$ and $g$ can be constant due to the degree restriction, we have arrived at a contradiction.
\end{proof}

\begin{lemma}\label{prehod}
The algebra $k[\xi][t]/(c(t))$ 
is isomorphic to the subalgebra $k[\xi][Y]$ of  the algebra $M_n(k[\xi])$.
\end{lemma}

\begin{proof}
Define a homomorphism from $k[\xi][t]/(c(t))$ to $k[\xi][Y]$ according to  $\xi_{ij} + (c(t))\mapsto \xi_{ij},\; t + (c(t))\mapsto Y$. 
Since $Y=(\xi_{ij})$ is a zero of its characteristic polynomial, this homomorphism is well-defined. It is obviously surjective.
  We can extend it to a homomorphism from $k(\xi)[t]/(c(t))$ to $M_n(k(\xi))$ which has trivial kernel as  $k(\xi)[t]/(c(t))$ is a field by Lemma \ref{ch}.  
\end{proof}

Recall that an integral domain $R$ is said to be integrally closed if it is integrally closed in its field of fractions $Q$. That is, if $f$ is a  monic polynomial  with coefficients in $R$, then every root of $f$ from $Q$ actually lies in  $R$.

\begin{proposition}\label{cz}
The algebra $k[\xi][Y]$ is integrally closed.
\end{proposition}

\begin{proof}
In view of Lemma \ref{prehod},   we have to show that $k[\xi][t]/(c(t))$  is integrally closed. To this end, it is enough to prove that  $B= k[\xi]/(\det(Y))$ is integrally closed. 
Indeed, let us set  $\xi'_{ii} = \xi_{ii} - t$ and  $\xi'_{ij} = \xi_{ij}$ for $i\ne j$. Then $Y-t$ is 
a generic matrix with entries $\xi'_{ij}$.  If $B$ is integrally closed, then so is 
$B' = k[\xi'_{ij}\mid 1 \le i, j \le n]/(\det(\xi'_{ij}))$, and therefore also $\K[t]/(c(t))=B'[t] = k[\xi_{ij}, t\mid 1 \le i, j \le n]/(\det(\xi'_{ij}))$.

Let $D = k[\eta_{ij}, \mu_{ji}\mid  1\leq i\leq  n, 1\leq j \leq n-1]$. We define the action of $GL_{n-1}$ on $D$ so that  $\sigma\in GL_{n-1}$  acts on an $n \times (n-1)$ matrix by $(a_{ij})^\sigma = (a_{ij})\sigma$,
and on an $(n-1) \times n$ matrix by $(b_{ij})^\sigma = \sigma^{-1}(b_{ij})$.
By the first fundamental theorem for the general linear group (see, e.g., \cite[Section 2.1.]{KP}), the algebra of invariants $D^{GL_{n-1}}$  is generated 
by the entries of the matrix $(\eta_{ij})(\mu_{ij})$. 
These entries parametrize the variety of matrices of rank at most $n-1$. Therefore the algebra $D^{GL_{n-1}}$ is isomorphic to the coordinate ring  $B= k[\xi]/(\det(Y))$ of this variety.
Thus, we can regard $B$ as a subalgebra $D^{GL_{n-1}}$ of $D$.  If $x$ lies in  the field of fractions of $B$
and is integral over $B$, then $x$ lies in the field of fractions of $D$ and is integral
over $D$. But $D$ is a polynomial algebra, so $x \in D$.  Hence
$x \in D^{GL_{n-1}}$, i.e.,
$x$ lies in the invariant subalgebra, namely $B$.
\end{proof}

\section{Commuting traces}\label{sec3}

Let $q:M_n(k)\to M_n(k)$ be the trace of a multilinear map. Identifying $M_n(k)$ with $k^{n^2}$, we can regard $q$ as a homogeneous polynomial map from $k^{n^2}$ to $k^{n^2}$.
The identity $[q(x),x]=0$ for every $x\in M_n(k)$ can be thus interpreted as $[q_\xi,Y]=0$ where  $Y$ is, as above, a  generic matrix, and $q_\xi$ is the matrix in $M_n(k[\xi])$ corresponding to $q$. In the next lemma we
 will show that the centralizer $\Ce(Y)$ of $Y$ in $M_n(\K)$ is trivial. As we shall see,  from this it can be easily derived   that  $q$ is of a standard form. We remark that if $d\ge n$, then 
a standard form $\sum_{i=0}^{d} \mu_i(x)x^i$ can be in 
    $M_n(k)$ written as $\sum_{i=0}^{n-1} \mu'_i(x)x^i$. This follows from the Cayley-Hamilton theorem.


\begin{lemma}\label{cent}
 $ \Ce(Y)= \K[Y]$. 
\end{lemma}

\begin{proof}
Since $Y$ is a generic matrix, it has $n$ distinct eigenvalues (see, e.g., \cite[Lemma 1.3.12]{Row}). Its centralizer in $M_n(\CC)$ is therefore the $\CC$-linear span of its powers. Thus, given  $a\in \Ce(Y)$, we have $a\in \CC[Y]$. 
By Lemma \ref{prehod}, 
$\K[Y]$ is integrally closed. Since $a$ is integral over $\K$ by the Cayley-Hamilton theorem and  $\CC[Y]$ is the field of fractions of $\K[Y]$, $a$ thus belongs to $\K[Y]$. This shows that 
 $ \Ce(Y)\subseteq \K[Y]$. The converse inclusion is trivial. 
 \end{proof}

\begin{lemma}\label{comtr}
Every commuting trace of a multilinear map $q:M_n(k)\to M_n(k)$  is of a standard form.
\end{lemma}

\begin{proof}
Since $q_\xi\in   \Ce(Y)$, Lemma \ref{cent} shows that  $q_\xi=\sum_{i=0}^{n-1} p_i(\xi_{11},\xi_{12},\dots,\xi_{nn})Y^i$  for some polynomials $p_i\in k[\xi]$. Since $q$ is the trace of a $d$-linear map  (for some $d\in \N$), $q_\xi$ is homogeneous of degree $d$ and therefore $\deg(p_i)=d-i$ (if $d<i$, then we  simply take $p_i=0$). Note that this is just another way of saying that $p_i$ is the trace of a $(d-i)$-linear central map. Hence $q$ is of a standard form.
\end{proof}

In the proof of the next theorem we combine Lemma \ref{comtr} with results from \cite{LLWW}. 
\begin{theorem}\label{krce}
If $A$ is a centrally closed prime $k$-algebra and $q:A\to A$ is a commuting trace of a multilinear map, then $q$ is of a standard form.
\end{theorem}

\begin{proof}  Let $M:A^d\to A$ be $d$-linear map such that $q(x)=M(x,\ldots,x)$, $x\in A$. In light of  \cite[Theorem 3.1]{LLWW} we may assume that 
 $A$ satisfies the standard polynomial identity $S_{2d}$. Since $A$ is assumed to be centrally closed, this simply  means that  $A$ is a  central simple algebra over  $k$ of dimension at most $d^2$ (see, e.g., \cite[Theorem C.2]{FI}).
Let $K\supseteq k$ be a splitting field of $A$. The scalar extension
$A_K=A\otimes_k K$ of $A$ to $K$ is then isomorphic to $M_n(K)$ for some $n\le d$.
A standard argument shows that  $M$ can be extended to  a $d$-linear (with respect to $K$)  map $\tilde{M}:A_K^d\to A_K$ (so that $\tilde{M}(x_1\otimes 1,\ldots,x_d\otimes 1) = M(x_1,\ldots,x_d)\otimes 1$).
 Let  $\tilde{q}(x):=\B(x,\ldots,x)$ be its trace. 
  It is easy to verify that $\tilde{q}$ is commuting, and thus, by Lemma  \ref{comtr}, of a standard form $\tilde{q}(x)=\sum_{i=0}^{n-1} \tilde{\mu}_i(x)x^i$, where $\tilde{\mu}_i$ is the trace of a $(d-i)$-linear map from $A_K$ to $K$ (i.e., a homogeneous polynomial of degree $d-i$). 
It remains to show that $\tilde{\mu}_i(A\otimes 1)\subseteq k$ to conclude that the restriction $M$ of $\B$ to $A\otimes 1$ is of a standard form.  
By \cite[Theorem 1.4]{LLWW}, $q(x)$ lies in the $k$-linear span of $1,x,\dots,x^{n-1}$. If the minimal polynomial of $a\in A$ has degree $n$, then it follows that $\tilde{\mu}_i(a)\in k$. 
We now follow the argument from the proof of  \cite[Theorem 3.1.49]{Row}. Let $h(x_1,\dots,x_n)=C_{2n-1}(1,x_1,\dots,x_1^{n-1},x_2,\dots,x_n)$, where $C_{2n-1}$ is the Capelli polynomial. If $h(a,x_2,\dots,x_{n})$ does not vanish on $A$, then the minimal polynomial of $a$ has degree $n$ (see, e.g., \cite[Theorem 1.4.34]{Row}),  and so in this case $\tilde{\mu}_i(a)\in k$.
Since $h$ is not an identity of $A$ (see, e.g., \cite[Proposition 3.1.6]{Row}), there exists $a\in A$ such that $h(a,x_2,\dots,x_n)$ is not a generalized polynomial identity. Take  $b\in A$. A standard Vandermonde argument shows that $h(a+\alpha b,x_2,\dots,x_n)$ is a generalized polynomial identity for only 
finitely many $\alpha\in k$. Therefore $\tilde{\mu}_i(a+\alpha b)\in k$ for infinitely many $\alpha\in k$,  and hence, again by a Vandermonde argument,  $\tilde{\mu}_i(b)\in k$.
\end{proof}

\section{Reduction to the commuting trace case} \label{s4}

The goal of this section is to prove a lemma which will be used as an essential tool for reducing the study of general functional identities of one variable to 
 commuting traces of multilinear maps. It considers general unital algebras,  so perhaps it could be  of use elsewhere.

Let $A$ be a unital algebra over  $k$. If $q:A\to A$ is the trace of a $d$-linear map $M:A^d\to A$, then we may assume that $M$ is symmetric (otherwise we replace it by the map $ \frac{1}{d!} \sum_{\sigma\in S_d} M(x_{\sigma(1)},\ldots,x_{\sigma(d)})$).  Then $M$ is unique, which makes it possible for us to define  $\partial q:A\to A$ by 
$$
\partial q(x) = M(x,\ldots,x,1).
$$
 Obviously, $\partial q$ is the trace of the symmetric $(d-1)$-linear map
$M(x_1,\ldots,x_{d-1},1)$.
Alternatively, one can introduce $\partial q$ by avoiding $M$ as follows. First notice that $q(x+\lambda)$, where $x\in A$ and $\lambda \in k$, is a polynomial in $\lambda$ with coefficients in $A$ of degree at most $d$. The coefficient at $\lambda$ is equal to  $d\partial q(x)$. In fact, $q(x+\lambda) = q(x) +  d \partial q(x) \lambda + \cdots +q(1) \lambda^d$ (and therefore, 
in a suitable setting, we have $\partial q(x) = \frac{1}{d}\lim_{\lambda\to 0}\frac{q(x+\lambda) - q(x)}{\lambda}$).
Using this alternative definition, one easily checks that if  $q'$ is another trace of, say,  $d'$-linear map, then we have 
\begin{equation}\label{od}
\partial \bigl(q(x)\cdot q'(x)\bigr) =\frac{d}{d+d'}\partial q(x)\cdot q'(x) + \frac{d'}{d+d'}q(x)\cdot   \partial q'(x).
\end{equation}
 
 In the next lemma we consider the situation where the traces of $d$-linear maps $q$ and $r$ satisfy $[q(x)x- xr(x),x]=0$, in other words, the situation where  $x\mapsto q(x)x - xr(x)$ is a commuting map.  We are actually interested in the special case where $q(x)x - xr(x)\in k$. However,  
 the proof runs more smoothly if we consider a slightly more general situation.

\begin{lemma}\label{L2}
Let $A$ be a unital $k$-algebra and
let $q,r:A\to A$  be the traces of $d$-linear maps. If $[q(x)x- xr(x),x]=0$ for all $x\in A$, then  there exists the trace of a $(d-1)$-linear map $p:A\to A$ such that 
$[[q(x) - xp(x),x],x]=0$ for all $x\in A$.
\end{lemma}

\begin{proof}
We set $q_0 = q$ and $q_t = \partial q_{t-1}$, $t=1,\ldots,d$. If $M$ is as above, then we have $q_t(x) = M(x,\ldots,x,1,\ldots,1)$ where $x$ appears $d-t$ times and $1$ appears $t$ times. 
 Analogously we introduce $r_t$, $t=0,\ldots, d$. For the traces of  multilinear maps $f$ and $g$ we will write $f(x)\equiv g(x)$ if $[f(x)-g(x),x]=0$ for every $x\in A$ (i.e., $f-g$ is commuting). 
 In this case we also have  $\partial f(x)\equiv \partial g(x)$.  Indeed, this follows immediately by applying \eqref{od} to $f(x)x-g(x)x=xf(x)-xg(x)$.
 Accordingly, 
 $q(x)x\equiv xr(x)$ implies $dq_1(x)x + q_0(x) \equiv dxr_1(x)  + r_0(x)$, and, furthemore, by induction one easily checks that
 \begin{equation}\label{vi2}
 (d-t)q_{t+1}(x)x + (t+1)q_t(x)   \equiv (d-t)xr_{t+1}(x) +  (t+1)r_t(x) 
\end{equation}
for $t=0,1,\ldots,d-1$. Next, we claim that
 \begin{equation}\label{g2}
 [q(x),x] \equiv (-1)^{t-1}\binom{d}{t}x(q_t - r_t)(x)x^t + \sum_{i=1}^t (-1)^{i-1}\binom{d}{i}x[r_i(x),x]x^{i-1} 
 \end{equation}
 for $t=1,\ldots,d$. Indeed, from $q(x)x\equiv xr(x)$ we get $[q(x),x] \equiv x(r-q)(x)$, and hence, using \eqref{vi2} for $t=0$, $[q(x),x] \equiv dx\bigl(q_1(x)x-xr_1(x)\bigr)$. Consequently,
 $$[q(x),x] \equiv dx (q_1-r_1)(x)x  + dx[r_1(x),x],$$
  which proves \eqref{g2} for $t=1$. Now assume  that $1 < t < d$. Using  \eqref{vi2} we obtain
 \begin{eqnarray*}
 &&(-1)^{t-1}\binom{d}{t}x(q_t - r_t)(x)x^t \\
  &\equiv& (-1)^{t-1}\binom{d}{t}\frac{d-t}{t+1} x \bigl(xr_{t+1}(x) - q_{t+1}(x)x\bigr)x^t\\
  &=&  (-1)^{t}\binom{d}{t+1}x[r_{t+1}(x),x]x^t + (-1)^{t}\binom{d}{t+1} x(q_{t+1} - r_{t+1})(x)x^{t+1}.
 \end{eqnarray*} 
This readily implies \eqref{g2} for any $t$.

Note that $q_d(x) = q(1)$ and $r_d(x) = r(1)$. Thus, the $t=d-1$ case of \eqref{vi2} reads as $q(1)x + dq_{d-1}(x) \equiv xr(1) + dr_{d-1}(x)$. Substituting $1$ for $x$ in this relation we obtain
$(d+1)q(1)\equiv (d+1)r(1)$. That is, $q_d(x)\equiv r_d(x)$, and so the $t=d$ case of \eqref{g2} reduces to 
$$
 [q(x),x] \equiv \sum_{i=1}^d (-1)^{i-1}\binom{d}{i}x[r_i(x),x]x^{i-1}.
$$
Since $x[r_i(x),x]x^{i-1} = [xr_i(x)x^{i-1},x]$, we can rewrite this as  
$$
 [q(x),x] \equiv \Bigl[\sum_{i=1}^d (-1)^{i-1}\binom{d}{i}xr_i(x)x^{i-1},x\Bigr]. 
$$
Thus, 
$$
p(x) := \sum_{i=1}^d (-1)^{i-1}\binom{d}{i}r_i(x)x^{i-1}
$$
satisfies $[[q(x) - xp(x),x],x] =0$. Note also that $p$ is indeed the trace of a $(d-1)$-linear map.
\end{proof}

\section{An Engel condition with traces of multilinear maps}

Lemma \ref{L2} gives rise to the question what can be said about  the trace of a multilinear map $s$ if it satisfies $[[s(x),x],x]=0$. Fortunately, an answer comes  easily if we combine our arguments with  an old lemma by Jacobson \cite{Jac} saying that  if $a,b\in M_n(k)$ are such that $[[b,a],a]=0$, then $[b,a]$ is nilpotent.

\begin{proposition}\label{pj} Let $A$ be a centrally closed prime $k$-algebra.
If $s:A\to A$ is the trace of a multinear map such that $[[s(x),x],x]=0$ for all $x\in A$, then $[s(x),x]=0$ for all $x\in  A$ (and hence $s$ is of a standard form
$s(x)= \sum_i \mu_i(x)x^i$).
\end{proposition}

\begin{proof}
If $A$ is infinite dimensional, then this  follows easily from the general theory of functional identities \cite{FI}; more explicitly,  it is a special case of \cite[Theorem 4.6]{Bei}. We may therefore assume that $A$ is finite dimensional, in which case it is a central simple algebra. A standard scalar extension argument shows that with no loss of generality we may assume that $A=M_n(k)$ (cf. the proof of Theorem \ref{krce}). Jacobson's lemma \cite[Lemma 2]{Jac} implies that $\hat{s}(x):= [s(x),x]$ is nilpotent for every $x\in A$.
As in Section \ref{sec3}, we may 
identify $\hat{s}$ with an element of $\K[Y]$. Since, by Lemmas \ref{ch} and \ref{prehod}, $\K[Y]$ is an integral domain, it does not contain nonzero nilpotents. 
Hence $\hat{s} =0$.
\end{proof}

In a seemingly more general, but in fact equivalent way we could formulate Proposition \ref{pj} so that the condition $[[s(x),x],x]=0$ is replaced by an Engel condition $[\ldots[[s(x),x],x],\ldots,x]=0$ (the number of brackets is arbitrary, but fixed). We remark that such a condition has been studied extensively  in the case where $s$ is a derivation or a related map (see, e.g., \cite{Lan}), and also in the case where $s$ is a general additive map \cite{BFLW, Bre95}. Proposition \ref{pj} thus takes a step further in this line of investigation.

\section{A functional identity determined by the adjugation}

Let $A$ be a central simple $k$-algebra with $\dim A = n ^2$. We define the trace of an element $a\in A$ as the trace of $a\tnz 1\in A\tnz K\iso M_n(K)$, where $K$ is a splitting field of $A$. 
It can be shown that the trace of $a$  belongs to $k$ and that this definition is independent of the choice of $K$ \cite[Theorem 3.1.49]{Row}. Hence,
 by using the fact that the coefficients of the characteristic 
polynomial of a matrix can be expressed through the traces of the powers of this matrix, one can also define  adj$(x)$, the {\em adjugate} of $x$. The map $x\mapsto {\rm adj}(x)$ is the trace of an $(n-1)$-linear map, in fact 
${\rm adj}(x) =   \sum_{i=0}^{n-1} \tau_i(x)x^i$ where $\tau_i$ is the trace of an $(n-1-i)$-linear map from $A$ into $k$. Just as for matrices, we have
$$
{\rm adj}(x^m) = {\rm adj}(x)^m
$$
and
$$
x{\rm adj}(x) = {\rm adj}(x)x = \det(x) \in k,
$$
where $\det(x)$ can be also defined through the traces of powers of $x$. The goal of this section is to show that the functional identity   $x^mq(x)\in k$  always arises from the adjugation. 

\begin{proposition}\label{thadj}
Let $A$ be a centrally closed prime  $k$-algebra, let $q\ne 0$ be the trace of a $d$-linear map, and let $m\in\N$. If $x^mq(x)\in k$ for all $x\in A$, then $A$ is finite dimensional. Moreover, if $\dim A = n^2$, then  $d\ge m(n-1)$ and there exists the trace of a  $(d-m(n-1))$-linear map $\lambda :A\to k$,  $\lambda\ne 0$, such that
$q(x)=\lambda(x){\rm adj}(x^m)$ for all $x\in A$. (In particular,   $x^mq(x)\ne 0$ for some $x\in A$.)
\end{proposition}

\begin{proof}
We first treat the case where $A\iso M_n(k)$.
 Then,  as in  Section \ref{sec3},  $q$ can be identified with $q_\xi\in M_n(k[\xi])$  and the condition $q(x)x^m\in k$ for every $x\in k$ reads as $q_\xi Y^m=\alpha$ where $\alpha\in k[\xi]$. 
Multiplying  by $\adj(Y^m)$ from the right we obtain 
\begin{equation}\label{ee}
\det(Y)^mq_\xi=\alpha {\rm adj}(Y^m).
\end{equation}
 Thus, $q_\xi\in k(\xi)[Y]$. Since  $q_\xi$ is integral over $k[\xi]$ 
by the Cayley-Hamilton theorem, and since   $k[\xi][Y]$ is integrally closed by Proposition \ref{cz}, it follows that  $q_\xi\in k[\xi][Y]$. Using
the Cayley-Hamilton theorem again,  we can write $q_\xi=\sum_{i=0}^{n-1} \mu_iY^i$ and ${\rm adj}(Y^m)=\sum_{i=0}^{n-1}\nu_iY^i$ with $\mu_i,\nu_i\in k[\xi]$. From \eqref{ee} we get
$\sum_{i=0}^{n-1} \bigl(\det(Y)^m \mu_i - \alpha\nu_i\bigr)Y^i =0$. Since there exist matrices whose degree of  algebraicity is $n$, and since  the linear dependence of matrices can be expressed  through zeros of a (Capelli) polynomial \cite[Theorem 1.4.34]{Row}, a Zariski topology argument shows that $\det(Y)^m \mu_i = \alpha\nu_i$ for each $i$.
Note that $\det(Y)$ does not divide $\nu_i$ for some $i$. Namely, otherwise  we would have  ${\rm adj}(Y^m)=\det(Y)r(Y)$ 
for some $r\in k[\xi][Y]$, implying that ${\rm adj}(A^m) =0$ whenever $A\in M_n(k)$ is not invertible. But one can easily find matrices for which this is not true (say, diagonal matrices with exactly one 0 on the diagonal).  Fix $i$ such that  $\det(Y)$ does not divide $\nu_i$.
Since $k[\xi]$ is a unique factorization domain and $\det(Y)$ is irreducible \cite{Berg},  $\det(Y)^m \mu_i = \alpha\nu_i$ implies  that  $\alpha=\det(Y)^m\lambda_0$ for some $\lambda_0\in k[\xi]$. From \eqref{ee} it follows that 
 $q_\xi=\lambda_0 {\rm adj}(Y^m)$. We can interpret this as $q(x)=\lambda(x){\rm adj}(x^m)$ for all $x\in A$ where $\lambda:A\to k$. Since ${\rm adj}(x^m)$ is the trace of a 
 an $m(n-1)$-linear map, 
 a simple homogeneity argument shows that $d\ge m(n-1)$ and that
  $\lambda$ is the trace of a  $(d-m(n-1))$-linear map. 

Consider now the general case where $A$ is an arbitrary centrally closed prime algebra. 
From the general theory of functional identities it follows that $A$ is finite dimensional (specifically, one can use \cite[Corollary 5.13]{FI}). 
From  now on we argue similarly as in the proof of Theorem \ref{krce}, so we  omit some details. 
Take a splitting field $K\supseteq k$ of $A$, and consider the  scalar extension
$A_K=A\otimes_k K\cong M_n(K)$. We extend $q$ to $\tilde{q}:A_K\to A_K$ (so that $\tilde{q}(x\otimes 1) = q(x)\otimes 1$) and observe that  $y^m\tilde{q}(y)\in K$ for all $y\in A_K$. Using the result of the previos paragraph we  readily arrive at the situation where  $q(x)\tnz 1={\rm adj} (x^m)\tnz\lambda(x)$ with $\lambda:A\to K$. If ${\rm adj}(x^m)\neq 0$ then $\lambda(x)\in k$;  since such elements $x$ exist, a Vandermonde argument easily yields  $\lambda(x)\in k$ for every $x\in A$.
\end{proof}

\section{General functional identities of one variable} \label{s7}

Our final theorem will be derived from  all the main results of the previous sections.

\begin{theorem}\label{MT} Let $A$ be a centrally closed prime $k$-algebra, and  let
$q_0,q_1,\ldots,q_m:A\to A$ be  the traces of $d$-linear maps. Suppose that 
$$q(x):=q_0(x)x^{m} + xq_1(x)x^{m-1} + \cdots + x^{m}q_{m}(x)\in k \,\,\,\mbox{for all $x\in A$.}$$
Then there exist the traces of $(d-1)$-linear maps $p_0,p_1,\ldots,p_{m-1}:A\to A$ and the traces of $d$-linear maps $\mu_0,\mu_1,\ldots,\mu_{m-1}:A\to k$ such that
\begin{align*}
q_0(x) &= x p_0(x) + \mu_0(x),\\
q_i(x) &= -p_{i-1}(x)x + xp_i(x) + \mu_i(x),\,\,\,i=1,\ldots,m-1,
\end{align*}
for all $x\in A$. Moreover:
\begin{enumerate}
\item[(a)] If $q(x)=0$ for all $x\in A$, then  
$$
q_m(x) = -p_{m-1}(x)x  -  \sum_{i=0}^{m-1} \mu_i(x)
$$
for all $x\in A$.
\item[(b)]  If $q(x)\ne 0$ for some $x\in A$, then $A$ is finite dimensional. If $\dim A = n^2$, then  $d\ge m(n-1)$ and there exists the trace of a  $(d-m(n-1))$-linear map $\lambda :A\to k$ such that $\lambda\ne 0$ and
$$q_m(x)=\lambda(x){\rm adj}(x^m)-p_{m-1}(x)x  - \sum_{i=0}^{m-1} \mu_i(x)$$ 
for all $x\in A$. 
\end{enumerate}
\end{theorem}

\begin{proof} Let us first  show that $q_0$ is of the desired form. Writing $q(x)$ as $(q_0(x)x^{m-1})x - x r(x)$, where $r(x) = -q_1(x)x^{m-1} - \cdots - x^{m-1}q_{m}(x)$, enables us  to  apply  
Lemma \ref{L2} and hence conclude that there exists the trace of a $(d+m-2)$-linear map $p:A\to A$ such that $[[q_0(x)x^{m-1} - xp(x),x],x]=0$ for all $x\in A$. Proposition \ref{pj} (together with Theorem \ref{krce}) therefore tells us that
$q_0(x)x^{m-1} - xp(x)=\sum_i \alpha_i(x)x^i$ with $\alpha_i:A\to k$. Rearranging the terms we can  write this  as $q_0(x)x^{m-1} - xr'(x)=\alpha_0(x)\in k$ where $r'$ is another  trace of a $(d+m-2)$-linear map. If $m=1$, then we are done. Otherwise, intrepret the last relation as  $(q_0(x)x^{m-2})x - xr'(x)\in k$, and repeat the above argument. In a finite number of steps we arrive at 
$q_0(x) = x p_0(x) + \mu_0(x)$.

We can now rewrite $q(x)\in k$ as
$$
x \Bigl( \bigl(p_0(x)x + q_1(x)\bigr)x^{m-1}  + xq_2(x)x^{m-2}+\cdots + x^{m-1}\bigl(q_{m}(x)+\mu_0(x)\bigr)\Bigr)\in k.
$$
From Proposition \ref{thadj} it follows that 
$$
\bigl(p_0(x)x + q_1(x)\bigr)x^{m-1}  + xq_2(x)x^{m-2}+\cdots + x^{m-1}\bigl(q_{m}(x)+\mu_0(x)\bigr)  = \sum \beta_i(x)x^i
$$
for some traces of multilinear maps $\beta_i:A\to k$  -- in fact, the right-hand side is either $0$ or it can be expressed through adj$(x)$, but the form we have stated is sufficient for our purposes. Namely, we can interpret the last identity as 
$$
\Bigr(\bigl(p_0(x)x + q_1(x)\bigr)x^{m-2}\Bigr)x - xs(x)=
\bigl(p_0(x)x + q_1(x)\bigr)x^{m-1} - xs(x)=\beta_0(x)\in k
$$
for some trace of a multilinear map $s$,
which brings us to the situation considered in the first paragraph. The same argument therefore gives $p_0(x)x + q_1(x) = xp_1(x) + \mu_1(x)$, i.e., $q_1$ is of the desired form. We can now continue this line of argumentation. In the next step we write $q(x)\in k$ as 
$$
x^2 \Bigl( \bigl(p_1(x)x + q_2(x)\bigr)x^{m-2}  + xq_3(x)x^{m-3}+\cdots + x^{m-2}\bigl(q_{m}(x)+\mu_0(x) + \mu_1(x)\bigr)\Bigr)\in k.
$$
First applying Proposition \ref{thadj} and then the argument from the first paragraph it thus follows that $q_2$ has the desired form. In this manner we describe the form of  all $q_i$ with $i < m$.  Cosequently, $q(x)\in k$ can be written as 
$$
x^m \Bigl(p_{m-1}(x)x + q_m(x) + \mu_0(x) + \cdots +\mu_{m-1}(x)\Bigr)\in k.
$$
The final conclusion now follows immediately from
Proposition \ref{thadj}.
\end{proof}

In the special case where  $d=1$ and $q(x)=0$ this theorem was proved by Beidar \cite[Theorem 4.4]{Bei}. In fact, he considered general prime rings and additive  maps $q_i$. The problem to extend Theorem \ref{MT} to prime rings and multiadditive maps remains open.

\bigskip

{\bf Acknowledgement.} The authors are greatly indebted to Irena Swanson for her substantial help in establishing the results of Section \ref{s2}.

\end{document}